\documentclass[11pt,a4paper]{amsart}

\newtheorem{theorem}{Theorem}

\newtheorem{lemma}[theorem]{Lemma}

\theoremstyle{definition}

\newtheorem{remark}[theorem]{Remark}

\begin{document}
\title[]{A note on the reachability of a Fibonacci control system}
\author[]{Anna Chiara Lai }
\address{Dipartimento  di Matematica e Fisica\\
Universita' degli Studi di Roma Tre\\
Largo San Murialdo, 1\\ 00146
 Roma}
 
\email{aclai@mat.uniroma3.it}


\maketitle \thispagestyle{empty}

%
\begin{abstract} 
Motivated by applications in robotics, we investigate a discrete control system related Fibonacci sequence and we characterize its reachable set. 
%
\end{abstract}

This note is devoted to the characterization of the reachable set of the discrete control system
\begin{equation*}\label{F}\tag{F}
 \begin{cases}
  x_0=u_0\\
  x_1=u_1+\frac{u_0}{q}\\
  x_{n+2}=u_{n+2}+\frac{x_{n+1}}{q}+\frac{x_{n}}{q^2}.
 \end{cases}
\end{equation*}
Motivations in investigating the systems of the form (F) come from robotics, indeed it can be shown that $x_n$ represents the
total length of a telescopic, self-similar robotic arm \cite{LLV14,LLV15}. 

 By an inductive argument we get the closed formula 
$$x_n=x_n(u)=\sum_{k=0}^{n} \frac{f_k}{q^k}u_{n-k},$$
where $f_{k}=f_{k-1}+f_{k-2}, f_1=f_0=1$ denotes Fibonacci sequence \cite{LLV15}. Consequently the \emph{asymptotic reachable set} $\mathcal R_q$ of the system (F) reads
 \begin{equation}\label{r}
  \mathcal R_q:=\{\lim_{n\to\infty} x_n(u)\mid u\in\{0,1\}^\infty\}=\left\{\sum_{k=0}^\infty \frac{f_k}{q^k}u_k\mid u_k\in\{0,1\}\right\}.
 \end{equation}
 
  \begin{remark}
  The set $\mathcal R_q$ is well defined if and only if the scaling ratio $q$ is greater than the 
  Golden Mean $\varphi$, this indeed ensures the convergence of the series $\sum_{k=0}^\infty \frac{f_k}{q^k}u_k$.
  \end{remark}
 
%
%

In order to give a full description of $\mathcal R_q$, we shall make use of the following definitions
\begin{align*}
\mathcal R_{q,j}&:=\left\{\sum_{k=0}^\infty \frac{f_{j+k}}{q^k}u_k\mid u_k\in\{0,1\}\right\}\\
S(q,j)&:=\sum_{k=0}^\infty \frac{f_{j+k}}{q^k}=\frac{f_{j}q^2+f_{j-1}q}{q^2-q-1}\\
Q(j)&:= \text{greatest solution of the equation $S(q,j+1)=q f_{j}$}\\
&= \frac{1}{2 f_j}(f_{j+2}+\sqrt{f_{j+2}^2+8 f_j^2}).
\end{align*}
Notice that, as $j\to\infty$, for all $q>\varphi$, $S(q,j)\uparrow \infty$ while $Q(j)\uparrow \frac{1}{2}(\varphi^2+\sqrt{\varphi^2+1})$. Also notice the recursive
relation 
\begin{equation}\label{S2}
 S(q,j)= q (S(q,j-1) - f_{j-1}).
\end{equation}

\begin{lemma}\label{reachability}
If $q\in(\varphi, Q(j)]$ then $\mathcal R_{q,j}=[0,S(q,j)]$. 
\end{lemma}
\begin{proof}
We show the claim by double inclusion. The inclusion $\mathcal R_{q,j}\subseteq[0,S(q,j)]$ readily follows by the definitions of $\mathcal R_{q,j}$ and of $S(q,j)$. 
To show the other inclusion, for all  $x\in[0,S(q,j)]$ we consider the sequences $(r_h)$ and $(u_h)$ defined by
\begin{equation}\label{rh}
 \begin{cases}
  r_0=x;\\
  u_h=\begin{cases}
1 &\text{if } r_h\in[f_{j+h},S(q,j+h)]\\
0 &\text{otherwise }
 \end{cases}\\
r_{h+1}=q(r_h-u_h f_{j+h})
\end{cases}
\end{equation}

We show by induction 
\begin{equation}\label{remainder}
 x=\sum_{k=0}^h \frac{f_{j+k}}{q^k}u_k + \frac{r_{h+1}}{q^{h+1}} \quad \text{ for all } h\geq 0.
\end{equation}
  For $h=0$ one has $r_1=q(x-u_0f_k)$ and consequently $x=f_ku_0+r_1/q$. Assume now (\ref{remainder}) as inductive hypothesis. Then
$$r_{h+2}=q^{h+2} \left( x- \sum_{k=0}^h \frac{f_{j+k}}{q^k}u_k\right) - q f_{j+h+1}u_{h+1}$$
Consequently,
$$x=\sum_{k=0}^{h+1} \frac{f_{j+k}}{q^k}u_k + \frac{r_{h+2}}{q^{h+2}}$$
and this completes the proof of the inductive step and, therefore, of (\ref{remainder}).

Now we claim that  if $q\leq Q(j)$ then
\begin{equation}\label{bound}
 r_h\in[0,S(q,j+h)]\quad \text{for every $h$}.
\end{equation}
We show the above inclusion by induction. If $h=0$ then the claim follows by the definition of $r_0$ and by the fact that $x\in[0,S(q,j)]$. Assume now (\ref{bound}) as inductive hypothesis and 
notice that the definition of $S(q,j+h)$ implies $f_{j+h}\leq S(q,j+h)$. If $r_h\in [0,f_{j+h})$ then $r_{h+1}=q r_h\in [0,q f_{j+h}]\subseteq [0, S(q,j+h+1)]$ -- 
where the last inclusion follows by the definition of $Q(j+h)$ and by the fact that $q\leq Q(j)<Q(j+h)$. 
If otherwise $r_h\in [f_{j+h},S(q,j+h)]$ then $r_{h+1}=q(r_h-f_{j+h})\subseteq [0, q(S(q,j+h)-f_{j+h})]=[0, S(q,j+h+1)]$ (see (\ref{S2})) and this completes the proof of (\ref{bound}).

Recalling $f_n\sim \varphi^n$ as $n\to \infty$,  one has
\begin{align*}
\sum_{k=0}^\infty \frac{f_{j+k}}{q^k}u_k&=\lim_{h\to\infty} \sum_{k=0}^{h-1} \frac{f_{j+k}}{q^k}u_k\stackrel{(\ref{remainder})}{=} x-\lim_{h\to\infty}\frac{r_h}{q^h}
\stackrel{(\ref{bound})}{\geq} x-\lim_{h\to\infty} \frac{S(q,j+h)}{q^h}\\&= x-\lim_{h\to\infty} \frac{q^2 f_{h+j+1}+ q f_{j+h}}{q^{j+h}(q^2-q-1)}=x.
\end{align*}
On the other hand
$$\sum_{k=0}^\infty \frac{f_k}{q^k}u_k=x-\lim_{h\to\infty} \frac{r_h}{q^h}\leq x$$
and this proves $x=\sum_{h=0}^\infty \frac{f_{h+k}}{q^h} u_h$. It follows by the arbitrariety of $x$ that $[0,S(q,k)]\subseteq \mathcal R_{q,k}$ and this concludes the proof. 
\end{proof}

\begin{remark}
By applying Lemma \ref{reachability} to the case $j=0$ we get that if $q\in (\varphi,Q(0)]$ then $\mathcal R_q=[0,S(q,0)]$.
This result was already proved in \cite{LLV15}.
\end{remark}

\begin{theorem}
  For all $j\geq 1$ if $q\in (Q(j-1),Q(j)]$ then  $\mathcal R_q$ is composed by the disjoint union of $2^j$ intervals,  in particular
 \begin{equation}\label{eq}
 \mathcal R_q=\bigcup_{u_0,\dots,u_{j-1}\in \{0,1\}} \left[ \sum_{k=0}^{j-1} \frac{f_k}{q^k} u_k, \sum_{k=0}^{j-1} \frac{f_k}{q^k} u_k+S(q,j)\right]. 
 \end{equation}
 Moreover if $q\geq \frac{1}{2}(\varphi^2+\sqrt{\varphi^2+1})$ then the map $u\mapsto x_u=\sum_{k=0}^\infty \frac{f_k}{q^k}u_k$ is increasing with respect to the lexicographic order and $\mathcal R_q$ is a totally disconnected set. 
\end{theorem}
\begin{proof}
 Fix $j\geq 1$ and let $q\in (Q(j-1),Q(j)]$.
 First of all we notice that 
 $$\mathcal R_q=\left\{\sum_{k=0}^\infty \frac{f_k}{q^k}u_k\mid u_k\in\{0,1\}\right\}=\bigcup_{u_0,\dots,u_{j-1}\in \{0,1\}}\sum_{k=0}^{j-1} \frac{f_k}{q^k} u_k+ \frac{1}{q^j}\mathcal R_{q,j}.$$
 Since $q\leq Q(j)$ then by Lemma \ref{reachability} we have $R_{q,j}=[0,S(q,j)]$ and this implies (\ref{eq}). We now want to prove that
 the union in (\ref{eq}) is disjoint.  To this end consider two binary sequences  $(v_0,\dots,v_{j-1})$ and $(u_0,\dots,u_{j-1})$ and assume $(v_0,\dots,v_{j-1})>(u_0,\dots,u_{j-1})$ 
in the lexicographic order. Let $h\in\{0,\dots,j-1\}$ be the smallest integer such that $v_h=1$ and
 $u_h=0$. Then $q>Q(j-1)\geq Q(h)$ implies
 $$\sum_{k=0}^{j-1} \frac{f_k}{q^k}v_k-\left(\sum_{k=0}^{j-1} \frac{f_k}{q^k}u_k+\frac{S(q,j)}{q^j}\right)\geq \frac{f_{h}}{q^{h}}-\sum_{k=h+1}^{j-1} \frac{f_k}{q^k}+\frac{S(q,j)}{q^j}=
 \frac{f_{h}}{q^{h}}-\frac{S(q,h+1)}{q^{h+1}}>0$$
 and, consequently, that the union in (\ref{eq}) is disjoint. 
 To show the second part of the claim we assume $q\geq \frac{1}{2}(\varphi^2+\sqrt{\varphi^2+1})$ and we let $u=(u_0,\dots,u_n,\dots,)$ and $v=(v_0,\dots,u_n,\dots,)$ be two infinite binary sequences such that $v>u$ in the lexicographic order. As
 above let $h$ be the smallest integer such that $0=u_h<v_h=1$ and define $x_\nu=\sum_{k=0}^\infty\frac{f_k}{q^k}\nu_k$ with $\nu\in\{u,v\}$. One has 
\begin{equation}\label{order}
x_v-x_u= \frac{f_h}{q^h}+\sum_{k=h+1}^\infty\frac{f_k}{q^k}v_k-\sum_{k=h+1}^\infty\frac{f_k}{q^k}u_k\geq \frac{f_h}{q^h}-\frac{1}{q^{h+1}}S(q,h+1)>0 
\end{equation}
Indeed  $q\geq \frac{1}{2}(\varphi^2+\sqrt{\varphi^2+1})$ implies $q>Q(h)$ for all $h\geq 0$. This implies that the map $\nu\mapsto x_\nu$ is increasing with respect to the lexicographic order. 
As a consequence, for all $x_w\in \mathcal R_q$ such that $x_u<x_w<x_v$ one has $u<w<v$ in the
lexicographic order. In particular $w_j=u_j=v_j$ for $j=0,\dots,h-1$, $w_h=u_h=0$ and $(w_{h+1},\dots, w_{h+n},\dots)> (u_{h+1},\dots,u_{h+n},\dots)$. Therefore $x_w=\sum_{k=0}^{h-1}\frac{f_k}{q^k}u_k+\delta_w$ and $\delta_w=\sum_{k=h+1}^\infty \frac{f_k}{q^k}w_k\leq \frac{1}{q^{h+1}}S(q,h+1)$. On the other hand the last inequality in (\ref{order}) implies that we may choose some
$$\delta\in\left(\frac{1}{q^{h+1}}S(q,h+1), \frac{f_h}{q^h}\right)$$
and setting $x:=x_u+\delta$ we get $x_u<x<x_v$  and, in view of above reasoning, $x\not\in \mathcal R_q$. 
By the arbitrariety of $x_u$ and $x_v$ we deduce that $\mathcal R_q$ is a totally disconnected set and this completes the proof.
\end{proof}

\end{document}